\UseRawInputEncoding
\documentclass[12pt]{article}
\usepackage{amssymb,amsmath}
\usepackage{latexsym,bm}
\usepackage{dsfont}

\usepackage{graphicx,amssymb,mathrsfs,amsmath,latexsym,amsfonts,amsthm}
\usepackage{amsmath,amssymb,amsthm}
\usepackage{latexsym}
\usepackage{amssymb}
\usepackage{stmaryrd}
 \usepackage{float}
\usepackage{psfrag}
\usepackage{xcolor}
\usepackage{enumerate}
\usepackage{manfnt,mathrsfs}
\usepackage{epstopdf}

\makeatletter 
\@addtoreset{equation}{section}
\makeatother
\newtheorem{theorem}{Theorem}

\newtheorem{corollary}[theorem]{Corollary}

\newtheorem{lemma}[theorem]{Lemma}

\newtheorem{problem}[theorem]{Problem}
\newtheorem{con}[theorem]{Conjecture}
\def\adots{\mathinner{\mkern2mu\raise0pt\hbox{.}  
\mkern2mu\raise4pt\hbox{.}\mkern1mu
\raise7pt\vbox{\kern7pt\hbox{.}}\mkern1mu}}

\newcommand{\A}{\mathcal{A}}

\newcommand{\V}{\mathcal{V}}

\date{ }
\begin{document}
\title{A note on extremal digraphs containing at most $t$ walks of length $k$ with the same endpoints}
\author{ Zhenhua Lyu\thanks{ School of Science, Shenyang Aerospace University, Shenyang, 110136, China. (lyuzhh@outlook.com)}  } \maketitle

\begin{abstract}
Let $n,k,t$ be positive integers. What is the maximum number of arcs in a digraph on $n$ vertices in which there are at most $t$ distinct walks of length $k$ with the same endpoints? In this paper, we prove that the maximum number is equal to $n(n-1)/2$ and the extremal digraph are the transitive tournaments when $k\ge n-1\ge \max\{2t+1,2\left\lceil \sqrt{2t+9/4}+1/2\right\rceil+3\}$. Based on this result, we may determine the maximum numbers and the extremal digraphs for $k\ge \max\{2t+1,2\left\lceil \sqrt{2t+9/4}+1/2\right\rceil+3\}$ and $n$ is sufficiently large, which generalises the existing results. A conjecture is also presented.

\end{abstract}


\section{Introduction}
We discuss only finite simple digraphs (without multiple arcs but allowing loops). The terminology and notation is that of \cite{bm}, except as indicated. The number of the vertices of a digraph is its {\it order} and the number of the arcs its {\it size}. We abbreviate directed walks and directed cycles as walks and cycles, respectively. The {\it length} of a walk or cycle is its number of arcs. A $p$-cycle is a cycle of length $p$. Similarly, a $p$-walk is a walk of length $p$. Let $D=(\V,\A)$ be a digraph with vertex set $\V$ and arc set $\A$. The size of $D$ is denoted by $a(D)$. The outdegree and indegree of a vertex $u$, denoted by $d^+(u)$ and $d^-(u)$, is the number of arcs with tails and heads $u$, respectively. Denote by
$$N^+(u)=\{x\in \V|(u,x)\in \A\}\quad and \quad N^-(u)=\{x\in \V|(x,u)\in \A\}.$$
 For a set $X\subset \mathcal{V}$, we denote by $D[X]$ the subgraph of $D$ induced by $X$. For $u,v\in \V$, $uv$  denotes the arc from $u$ to $v$ and  the notation $u\rightarrow v$ means $uv\in \A$.

Tur\'an type problems are among the most important topics in graph theory, which concern the possible largest number of edges in graphs forbidding given subgraphs and the extremal graphs achieving that maximum number of edges.  The systematic investigation of digraph extremal problem was initiated by Brown and Harary \cite{BH}. For more details, see  \cite{BES,BS}. Given a family of digraphs $\mathscr{F}$, a digraph $D$ is said to be {\it $\mathscr{F}$-free} if $D$ contains no subgraph from $\mathscr{F}$.  Let $ex(n,\mathscr{F})$ be  the maximum size of $\mathscr{F}$-free  digraphs of order $n$ and $EX(n,\mathscr{F})$ be the set of $\mathscr{F}$-free digraphs of order $n$ with size $ex(n,\mathscr{F})$. Given two positive integers $k,t$, denote by $\mathscr{F}_{k,t}$ the family of  digraphs consisting of $t$ different walks of length $k$ with the same initial vertex and the same terminal vertex. In \cite{HL2}, the authors posed a Tur\'an type  problem as follows.
\begin{problem}\label{pro1}
Given positive integers $n,k,t$, determine $ex(n,\mathscr{F}_{k,t+1})$ and  $EX(n,\mathscr{F}_{k,t+1})$.
\end{problem}
The initial version of Problem \ref{pro1}  was posed by Zhan at a seminar in 2007, which concerned  the case $t=1$, see \cite[p.\,234]{ZH}. In the last decade,  Problem \ref{pro1} for the case $t=1$ has been completely solved by Wu \cite{WU}, by Huang and Zhan \cite{HZ1}, by Huang, Lyu and Qiao \cite{HLQ}, by Lyu \cite{lyu}, and by Huang and Lyu \cite{HL3}. For the general cases of Problem \ref{pro1}, the case $k=2$ has been studied in \cite{lyu20a}, and the case for $k\ge n-1\ge 6t+1$ has been solved in \cite{HL2}.
\begin{theorem}[\cite{HL2}]\label{th2} Let $t$ be a positive integer. For  $k\ge n-1\ge 6t+1$, a digraph $D\in EX(n,\mathscr{F}_{k,t+1})$ if and only if $D$ is a transitive tournament.
\end{theorem}
\noindent We define $z(t)$ as the smallest integer such that if $k\ge n-1\ge z(t)$, then $D\in  EX(n,\mathscr{F}_{k,t+1})$ if and only if $D$ is a transitive tournament. Huang and Zhan \cite{HZ1} proved that $z(1)=4$. It follows from Theorem \ref{th2} that $z(t)$ is well defined for each positive integer $t$ and
\begin{equation*}
z(t)\le 6t+1.
\end{equation*} 
Based on this fact, using induction on $n$, Lyu \cite{lyu21} obtained the following result.
\begin{theorem}\label{th3}
Let $k,n,t$ be positive integers with $k\ge 6t+1$ and $n\ge k+5+\lfloor\log_{2}(t)\rfloor$. Then $D\in EX(n,\mathscr{F}_{k,t+1})$ if and only if $D$ is an balanced blow-up of the transitive tournament of order $k$.
\end{theorem}
Motivated by Theorem \ref{th3}, \cite[Theorem 2]{lyu} and \cite[Theorem 1]{HLQ}, we present the conjecture as follows.
\begin{con}\label{c1}
Let $k\ge z(t)$ and let $n$ be sufficiently large. Then $D\in EX(n,\mathscr{F}_{k,t+1})$ if and only if $D$ is a balanced blow-up of the transitive tournament of order $k$.
\end{con}
From \cite{HZ1} we get $z(1)=4$. Hence, Conjecture \ref{c1} holds confirmly when $t=1$. In the point of this view, it is important to determine the exact value or a better upper bound of $z(t)$ for each $t$. In this note, we present a new upper bound for $z(t)$ as follows.
\begin{theorem}\label{th1}
let $t$ be a positive integer. Then
\begin{equation}
z(t)\le \max\{2t+1,2\left\lceil \sqrt{2t+9/4}+1/2\right\rceil+3\}.
\end{equation}
\end{theorem}
Adopting the same arguments as in the proofs in \cite{lyu21}(modify a few details in the proof), we may obtain that  Theorem \ref{th3} holds for $k\ge \max\{2t+1,2\left\lceil \sqrt{2t+9/4}+1/2\right\rceil+3\}$, which improves the main result of \cite{lyu21} when $t\ge 2$.

\section{Proof of Theorem \ref{th1}}
We need the following lemmas.

\begin{lemma}\label{le12}Let $n,t$ be positive integers and let $D$ be a digraph of order $n$.
If an $m_1$-cycle $C_{1}$ and an $m_2$-cycle $C_{2}$ in $D$  are joint, then $D$ is not $\mathscr{F}_{k,t+1}$-free for all $k\ge L\lceil\log_{2}(t+1)\rceil,$ where $L$ is the least common multiple of $m_1$ and $m_2$.
\end{lemma}
\begin{proof} Let $a_1=L/m_1$ and $a_2=L/m_2$. Assume $C_{1}$ and $C_{2}$ are joint at vertex $v$. Let $w$ be the walk of length $L\lceil\log_{2}(t+1)\rceil$ from $v$ to $v$ along $C_1$. We partition $w$ into  $\lceil\log_{2}(t+1)\rceil$ walks of the same length from $v$ to $v$, say $w_1,\ldots,w_{\lceil\log_{2}(t+1)\rceil}$. Each of $\{w_1,\ldots,w_{\lceil\log_{2}(t+1)\rceil}\}$  could be replaced by repeating $C_{2}$ $a_2$ times. Therefore, there exist $t+1$ distinct walks of length $L\lceil\log_{2}(t+1)\rceil$ from $v$ to $v$. For $k> L\lceil\log_{2}(t+1)\rceil$, we can  extend these walks along $C_{2}$ to $k$-walks with the same endpoints.
\end{proof}

\begin{lemma}[\cite{HL2}]\label{le1}
Let $n,t$ be positive integers and let $D$ be a digraph of order $n$.
If an $m_1$-cycle $C_{1}$ and an $m_2$-cycle $C_{2}$ 	in $D$  are  connected by an arc, then $D$ is not $\mathscr{F}_{k,t+1}$-free for $k\ge  t L+1.$
\end{lemma}
The {\it girth} of a digraph with a cycle is the length of its shortest cycle, and a digraph with no cycle has infinite girth.
\begin{lemma}[\cite{HL2}]\label{le3}
Let $D$ be a loopless digraph of order $n$.  If $a(D)= n(n-1)/2$, then  $D$ is a transitive tournament or
	$$g(D)\le 3.$$
\end{lemma}

Let $D=(\V,\A)$ be a digraph with $l$ loops. Denote by $d(u)$ the number of arcs incident with $u$. We have
\begin{equation*}\label{eqh2}
d(u)=\left\{\begin{array}{ll}
                 d^+(u)+d^-(u)-1,& if ~u\rightarrow u;\\
                 d^+(u)+d^-(u),&\textrm{otherwise}.\end{array}\right.
                 \end{equation*}
Since $a(D)=\sum\limits_{u\in \V}d^+(u)$ and $a(D)=\sum\limits_{u\in \V}d^-(u)$, we have
\begin{equation}\label{e11}
2a(D)=\sum\limits_{u\in \V}d(u)+l,\end{equation}
and $$d(u)=a(D)-a(D[\V\setminus \{u\}])~\text{for~all}~u\in \V.$$
\begin{lemma}\label{le5}Let $D=(\V,\A)$ is $\mathscr{F}_{k,t+1}$-free with $k\ge 2\lceil\log_2(t+1)\rceil$. Then $d(u)\le |\V|$ for all $u\in \V$.
\end{lemma}
\begin{proof} Suppose there exists $u\in \V$ such that $d(u)\ge |\V|+1$. It follows that at least two cycles are joint at $u$. Moreover, these 2 cycles are either two 2-cycles or one loop and one 2-cycle. By Lemma \ref{le12}, $D$ is not $\mathscr{F}_{k,t+1}$-free for $k\ge 2\lceil\log_2(t+1)\rceil$, a contradiction.
\end{proof}

\begin{lemma}\label{le13}Let $D=(\V,\A)$ be a digraph and let $C$ and $T$ be disjoint cycle and tournament in $D$, where $|\V(T)|\ge 2\left\lceil \sqrt{2t+9/4}+1/2\right\rceil+1$. If there exists at least one arc between each vertex of $C_1$ and each vertex of $T$, then $D$ is not $\mathscr{F}_{k,t+1}$-free for $k\ge \max\{t+1,3\lceil\log_2(t+1)\rceil\}$.
\end{lemma}
\begin{proof} To the contrary, suppose $D$ is $\mathscr{F}_{k,t+1}$-free for  $k\ge \max\{t+1,3\lceil\log_2(t+1)\rceil\}$. Suppose $T$ contains a 3-cycle $C_1$ as its subdigraph. Let
$$ C\equiv w_1\rightarrow \dots\rightarrow  w_l\rightarrow w_1 ~\text{and}~ C_1\equiv u_1\rightarrow u_2\rightarrow u_3\rightarrow u_1.$$
Without loss of generality, we assume $w_1\rightarrow u_1$. If $u_2\rightarrow w_1$, we obtain a 3-cycle $u_2\rightarrow w_1\rightarrow u_1\rightarrow u_2$. Since two 3-cycles are joint, by Lemma \ref{le12} we obtain $D$ is not $\mathscr{F}_{k,t+1}$-free, a contradiction. Hence, $w_1\rightarrow u_2$. Similarly, $w_1\rightarrow u_3$. If there exists some $i$ such that $u_i\rightarrow w_l$, we obtain $u_i\rightarrow w_l\rightarrow w_1\rightarrow u_i$. Then two 3-cycles are joint. By Lemma \ref{le12}, $D$ is not $\mathscr{F}_{k,t+1}$-free, a contradiction. Hence $w_l\rightarrow u_i$ for $i\in \{1,2,3\}$. Repeate the above arguments, we have
\begin{equation}\label{e6}
w_i\rightarrow u_j\quad \text{for}~ i\in \{1,2,\ldots,l\} ~\text{and}~ j\in\{1,2,3\}.
\end{equation}

We construct walks of length $k$ from $w_1$ to $u_1$ in the following way. For each $t_1\in \{0,1,\ldots, t\}$,  there are a walk of length $t_1$ with its initial vertex $w_1$ along $C$, say $W_{t_1}$, and
a walk of length $k-t_1-1$ with terminal vertex $u_1$ along $C_1$, say $W'_{t_1}$.  Since (\ref{e6}), $W_{t_1}W'_{t_1}$ is a walk of length $k$ with initial vertex $w_1$ and terminal vertex $u_1$. Then there exist $t+1$ distinct walks of length $k$ from $w_1$ to $u_1$, a contradiction. Hence $T$ contains no 3-cycles. Combining this with Lemma \ref{le3},  $T$ is acyclic, and hence it is transitive. Let $a=\left\lceil \sqrt{2t+9/4}+1/2\right\rceil$. Since $|\V(T)|\ge 2a+1$, without loss of generality, we assume $w_1$ has at least $a+1$ successors in $T$. Let those successors be $\{t_0,t_1,t_2,\ldots,t_a\}$ with $t_i\rightarrow t_0$ for $i=\{1,2,\ldots, a\}$. For any pair $i,j\in \{1,2,\ldots, a\}$ with $i<j$, we have  $\cdots\rightarrow w_1\rightarrow t_i\rightarrow t_j\rightarrow t_0$. Since $a(a-1)/2\ge t+1$, then there are more than $t$ walks of length at least 3, a contradiction.
\end{proof}

Now we are ready to give the proof of Theorem \ref{th1}.

\noindent {\bf Proof of Theorem \ref{th1}}.
Let $a=2\left\lceil \sqrt{2t+9/4}+1/2\right\rceil+1$ and  let $n'\ge  a+2$. It is easily seen that the transitive tournament of order $n'$ is in $EX(n',\mathscr{F}_{k,t+1})$. Hence,
\begin{equation}
ex(n',\mathscr{F}_{k,t+1})\ge \frac{n'(n'-1)}{2}
\end{equation}
 First we prove that
\begin{equation}\label{e2}
ex(n',\mathscr{F}_{k,t+1})= \frac{n'(n'-1)}{2}.
 \end{equation}
Suppose otherwise that  $D$ is $\mathscr{F}_{k,t+1}$-free on $n'$ vertices with $l$ loops and
\begin{equation}\label{e1}a(D)\ge \frac{n'(n'-1)}{2}+1.
 \end{equation}
By the pigeonhole principle, there exists some $v$ such that $d(v)\ge n'$. Combining this with Lemma \ref{le5}, we have
\begin{equation}\label{e4}
d(v)= n'.
\end{equation}We distinguish the following two cases.

{\it Case 1.} $vv\notin \A$. Then there exists $u\in \V\setminus \{v\}$ such that $vu,uv\in \A$. By Lemma \ref{le12}, two $2$-cycles can not be joint. Hence $v$ is on exactly one 2-cycle. Combining this with (\ref{e4}), each vertex in $\V\setminus \{u,v\}$ is jointed with $v$ by exactly one arc.  By Lemma \ref{le1}, $D-v$ has no 2-cycles or loops, which implies that $d(w)\le n'-1$ for all $w\in \V\setminus \{v,u\}$. By Lemma \ref{le5} and (\ref{e1}), we obtain $d(u)=n'$ and
\begin{equation*}\label{e5}
d(w)= n'-1{\rm~ for~ all~} w\in \V\setminus \{v,u\}.
\end{equation*}
Hence, $D[\V\setminus \{u,v\}]$ is a tournament. Moreover, each vertex in $\V\setminus \{u,v\}$ is jointed with each of  $\{u,v\}$ by exactly one arc. Since $|\V\setminus \{u,v\}|\ge a$, by Lemma \ref{le13}, $D$ is not $\mathscr{F}_{k,t+1}$, a contradiction.

{\it Case 2.} $vv\in \A$. By Lemma \ref{le1}, $v$ is jointed by exactly one arc with each vertex in $\V$. Moreover, each vertex in $\V\setminus \{v\}$ is not on a loop or a 2-cycle. It follows from (\ref{e1}) and  (\ref{e4}) that $D[\V\setminus \{v\}]$ is a tournament. By Lemma \ref{le13}, $D$ is not $\mathscr{F}_{k,t+1}$-free, a contradiction. Now we get (\ref{e2}).

\indent Now we characterize the structures of the digraphs in $EX(n,\mathscr{F}_{k,t+1})$. Let $D\in EX(n,\mathscr{F}_{k,t+1})$. First we show that $D$ contains no loops. By (\ref{e2}), we obtain that $$a(D)=\frac{n(n-1)}{2}~\text{and}~a(D[\V\setminus \{u\}])\le (n-1)(n-2)/2.$$ Combining this with the definition of $d(u)$, we get $d(u)\ge n-1$ for all $u\in \V$. Recalling (\ref{e11}), $D$ is loopless. Moreover,
\begin{equation}\label{e3}d(u)= n-1~{\rm for~ all~} u\in \V.
\end{equation}

Suppose $D$ contains a 2-cycle $u\rightarrow v\rightarrow u$. By Lemma \ref{le12}, $v$ can not be on two distinct 2-cycles. Hence, it is joined with $n-2$ distinct vertices. Let $v_0\in \V$ such that there is no arc between $v$ and $v_0$. From (\ref{e3}) $v_0$ is on a 2-cycle, say $v_0\rightarrow v_1\rightarrow v_0$. Obviously, $v_1$ is joined with $v$. By Lemma \ref{le1}, $D$ is not $\mathscr{F}_{k,t+1}$-free, a contradiction. Hence, $D$ contains no 2-cycles. Recalling (\ref{e3}), $D$ is a tournament.

Suppose $D$ contains a 3-cycle $u\rightarrow v\rightarrow w\rightarrow u$. Note that $D[\V\setminus \{u,v,w\}]$ is also a tournament. Since $n-3\ge a$, by Lemma \ref{le13}, $D$ is not $\mathscr{F}_{k,t+1}$-free, a contradiction. It follows from Lemma \ref{le3} that $D$ is a transitive tournament.

Conversely, it is easily seen that  the transitive tournament of order $n$ is in $EX(n,\mathscr{F}_{k,t+1})$. This completes the proof.
\qed







\end{document}